\documentclass[12pt]{amsart}

\usepackage{fullpage}
\usepackage{mathtools}
\usepackage{amssymb,amsmath,amsthm,amscd,mathrsfs,graphicx}
\usepackage[dvipsnames]{xcolor}
\usepackage{bm}
\usepackage{hyperref}
\usepackage{dsfont}
\usepackage{enumerate}
\usepackage{epsfig}
\usepackage{float, graphicx}
\usepackage{latexsym, amsxtra}
\usepackage{pdfpages}
\usepackage{multicol}
\usepackage[normalem]{ulem}
\usepackage{psfrag}
\usepackage{stmaryrd}
\usepackage{tikz}
\usepackage[T1]{fontenc}
\usepackage{url}
\usepackage{verbatim}
\usepackage{xy}
\usepackage{indentfirst}
\usepackage{tikz-cd}
\usepackage{url}
\usepackage{cases}
\usepackage{amsmath}

\makeatletter
\def\thm@space@setup{%
  \thm@preskip=2ex \thm@postskip=2ex
}
\makeatother

\hypersetup{hidelinks}

\numberwithin{equation}{section}
\theoremstyle{plain}

\newtheorem{thm}{Theorem~}[section]
\newtheorem{lem}[thm]{Lemma~}

\newtheorem{prop}[thm]{Proposition~}

\theoremstyle{remark}

\theoremstyle{definition}
\newtheorem{defn}[thm]{Definition~}

\newcommand{\calP}{\mathcal{P}}

\newcommand{\CC}{\mathbb{C}}
\newcommand{\ZZ}{\mathbb{Z}}

\newcommand{\PP}{\mathbb{P}}

\newcommand{\calO}{\mathcal{O}}

\newcommand\PGL{\mathrm{PGL}}

\newcommand\diag{\mathrm{diag}}
\newcommand\GL{\mathrm{GL}}

\newcommand\Lin{\mathrm{Lin}}

\title{Abelian Automorphism Groups Of Quartic Surfaces and Cubic Fourfolds}

 \author[T. Peng]{Tianzhen Peng}
\address{Hangzhou Foreign Languages School, Hangzhou, China}
\email{brianpengtz@outlook.com}

 \author[Z. Zheng]{Zhiwei Zheng}
\address{Yanqi Lake Beijing Institute of Mathematical Sciences and Applications, Beijing, China}
\email{zhengzw11@163.com}

\date{}
\begin{document}
\bibliographystyle{amsalpha}

\begin{abstract}
In this paper, we develop a new method to classify abelian automorphism groups of hypersurfaces. We use this method to classify (Theorem \ref{theorem: main}) abelian groups that admit a liftable action on a smooth cubic fourfold. A parallel result (Theorem \ref{theorem: main2}) is obtained for quartic surfaces. 
\end{abstract}

\maketitle
\section{Introduction}
Cubic fourfolds are very important objects in algebraic geometry due to its close relation with rationality problems and hyper-K\"ahler geometry. In this paper, we mainly study abelian automorphism groups of smooth cubic fourfolds. Previously, Gonz\'alez-Aguilera and Liendo \cite{gonzalez2011automorphisms} classified automorphisms of smooth cubic fourfolds with prime order (which must be $2,3,5,7$ or $11$). Fu \cite{fu2016classification} classified symplectic automorphisms of smooth cubic fourfolds with primary (namely, power of a prime) orders. Here, an automorphism of a smooth cubic fourfold is called symplectic if the induced action on the Fano scheme of lines (which is a hyper-K\"ahler fourfold) is symplectic. Recently, Laza and the second author \cite{laza2019automorphisms} classified all symplectic automophism groups of smooth cubic fourfolds via Hodge theory and lattice theory. In this paper, we mainly consider abelian automorphism groups of cubic fourfolds. A classification of abelian automorphism groups of smooth cubic fourfolds is claimed in a long unpublished note \cite{2013abelian} (which has more than 200 pages). In this paper, we develop a new method to study abelian automorphism groups of hypersurfaces. Applying our method, we recover the classification of all liftable (see Definition \ref{definition: lift}) actions of abelian groups on smooth cubic fourfolds, see Theorem \ref{theorem: main}. Our proof is rather short and totally self-contained, which shows the efficiency of our method. Besides the case of cubic fourfolds, we also give the parallel result (Theorem \ref{theorem: main2}) for quartic surfaces.

\smallskip

In \cite{zheng2020abelian}, the second author studied the abelian automorphism groups of smooth hypersurfaces. The method we use in this paper to prove Theorem \ref{theorem: main} and \ref{theorem: main2} is an extension and refinement of the method in \cite{zheng2020abelian}. Our approach is based on the fact that a linear action of an abelian group on  $\CC^N$ can be diagonalized. Namely, for an abelian group $G$ and an injective homomoprhism $G\hookrightarrow \GL(N, \CC)$, we can always find an element $g\in \GL(N, \CC)$, such that $g G g^{-1}$ is contained in $(\CC^{\times})^N$. The application of this fact in the study of automorphisms of hypersurfaces originated from Gonz\'alez-Aguilera and Liendo \cite{gonzalez2013order}. They classified orders of automorphisms which are primary and coprime to $d(d-1)$.
\smallskip

One novelty of this paper is to associate a directed graph with a liftable action (of an abelian group $G$ on a smooth $(N-2)$-fold $V(F)$ of degree $d$). We sketch the construction of this graph. Suppose this action is diagonalized by a choice of coordinate system $x_1, \cdots, x_N$. Then $F$ can be written as a homogeneous polynomial of degree $d$ in $x_1, \cdots, x_N$. Take characters $\lambda_1, \cdots, \lambda_N, \lambda\colon G\to \CC^{\times}$ such that
\begin{equation*}
g=\diag(\lambda_1(g), \cdots, \lambda_N(g))\in (\CC^{\times})^{N}
\end{equation*}
and $F(\lambda_1(g) x_1, \cdots, \lambda_N(g)x_N)=\lambda(g)F(x_1,\cdots, x_N)$ for any $g\in G$. We write
$$I=(G, \lambda, \lambda_1, \cdots, \lambda_N).$$
We consider a set $S_I$ of monomials $m=\prod_{i=1}^N x_i^{\alpha_i}$  such that $\sum_{i=1}^N {\alpha_i}=d$ and $\lambda=\prod_{i=1}^N \lambda_i^{\alpha_i}$. Let $S(F)$ be the set of monomials with nonzero coefficients in $F$. Then $S(F)\subseteq S_I$. From the smoothness of $F$, we know that for any $x_i$, there exists $x_j$ such that $x_i^{d-1} x_j\in S(F)$ (see Lemma \ref{lemma: main criterion}). For a set $S$ of monomials in $x_1, \cdots, x_N$ of degree $d$, we define a directed graph with $[N]=\{1,\cdots,N\}$ the set of vertices, and $E_{S}=\{(i,j)|x_i^{d-1}x_j\in S\}$ the set of arrows (from $i$ to $j$). Therefore, if a homogeneous polynomial $F$ is smooth, every vertex of the graph $E_{S(F)}$ has at least one outlet. With a more detailed analysis of the structure of $E_{S_I}$, we will eventually prove Theorem \ref{theorem: main} and \ref{theorem: main2}.
\smallskip

\noindent \textbf{Stucture of the paper}: In \S \ref{section: pre} we introduce some terminologies and preliminary results that will be used throughout the paper. In \S \ref{section: cubic} we specialize to the study of automorphisms of cubic hypersurfaces. In \S \ref{section: cubic fourfold} we concentrate on the study of cubic fourfolds and prove Theorem \ref{theorem: main}. Finally in \S\ref{section: quartic} we consider the case of quartic surfaces.

\smallskip

\noindent \textbf{Acknowledgement}: This work is the outcome of the first author's PCP project in 2020. We thank the staff of PCP for their support. The second author thanks Max Planck Institute for Mathematics for its support and excellent academic atmosphere during his stay.

\section{Preliminary Results}
\label{section: pre}
Given integers $d\geqslant 3, N\geqslant 3$. In this section, we fix some terminologies about $(N-2)$-folds of degree $d$ and their automorphism groups. We denote by $V_N$ a complex vector space of dimension $N$, and denote by $\PP^{N-1}$ the projectivization of $V_N$. Let $\GL(N)$ be the group of linear automorphisms of $V_N$, and let $\PGL(N)$ be the group of linear automorphisms of $\PP^{N-1}$. There is a natural group homomorphism $\GL(N)\to\PGL(N)$.

For a global section $F\in H^0(\PP^{N-1},\calO(d))$, we denote by $V(F)\subseteq \PP^{N-1}$ the hypersurface of degree $d$ defined by the zero locus of $F$. Let $\Lin(V(F))$ be the group of linear automorphisms of $\PP^{N-1}$ preserving $V(F)$. By Matsumura-Monsky \cite{matsumura1963automorphisms}, the group $\Lin(V(F))$ is finite when $V(F)$ is smooth. Moreover, if $V(F)$ is smooth, $N\ge 4$ and $(d,N)\ne(4,4)$, then $\Lin(V(F))$ is equal to the group of regular automorphisms of $V(F)$ as a complex variety.

\begin{defn}
\label{definition: lift}
Let $V(F)\subseteq \PP^{N-1}$ be an $(N-2)$-fold of degree $d$. Take a subgroup $G\leqslant \Lin(V(F))\leqslant \PGL(N)$. The action of $G$ on $V(F)$ is called liftable, if the group embedding $G\hookrightarrow \PGL(N)$ factors through $\GL(N)\to \PGL(N)$ (Namely, there exists a group homomorphism $G\to \GL(N)$, such that the composition of $G\to \GL(N)\to \PGL(N)$ equals to the given one). Here, the group embedding $G\hookrightarrow \GL(N)$ is called a lifting of the action.
\end{defn}

Given a coordinate system $(x_1,\cdots, x_N)$, we call a linear automorphism of $V_N$ diagonal with respect to this coordinate system, if it sends $(x_1, \cdots, x_N)\in V_N$ to $(\lambda_1 x_1, \cdots, \lambda_N x_N)$ for certain $\lambda_i\in \CC^{\times}$. We denote such an automorphism by $\diag(\lambda_1, \cdots, \lambda_N)$, and denote by $\diag[\lambda_1:\cdots:\lambda_N]$ the equivalence class in $\PGL(N)$.

\smallskip

\subsection{Poset $\calP$}
\label{subsection: poset}
Now we take a finite abelian group $G$. Suppose we have a liftable action of $G$ on a smooth $(N-2)$-fold $V(F)$ of degree $d$. We choose a lifting $\varphi_1\colon G\hookrightarrow \GL(N)$. Since $G$ is abelian, by standard linear algebra, we can choose a coordinate system $(x_1, \cdots, x_N)$ of $V_N$, such that all elements in $G$ is diagonal with respect to $(x_1, \cdots, x_N)$. Therefore, we have $N$ characters $\lambda_1, \cdots, \lambda_N\colon G\to \CC^{\times}$, such that $g(x_1, \cdots, x_N)=(\lambda_1(g) x_1,\cdots, \lambda_N(g)x_N)$ for any $g\in G$.

\smallskip

Since the action of G on $\PP^{N-1}$ preserves the hypersurface $V(F)$, there exists a character $\lambda\colon G\to \CC^{\times}$, such that $F(\lambda_1(g)x_1, \cdots, \lambda_N(g)x_N)=\lambda(g) F(x_1,\cdots, x_N)$ for any $g\in G$ and $x_1, \cdots, x_N\in \CC$. We denote by $I$ the data $(G, \lambda, \lambda_1, \cdots, \lambda_N)$ and define

\begin{equation*}
S_I\coloneqq\{m=\prod_{i=1}^N x_i^{\alpha_i}\big{|}\sum_{i=1}^N {\alpha_i}=d, \lambda=\prod_{i=1}^N \lambda_i^{\alpha_i}\}.
\end{equation*}

Suppose we have two liftings $\varphi, \varphi'\colon G\hookrightarrow \GL(N, \CC)$, giving rise to $I=(G, \lambda, \lambda_1, \cdots, \lambda_N)$ and $I'=(G, \lambda', \lambda_1', \cdots, \lambda_N')$. Then there exists a character $\xi\colon G\to \CC^{\times}$, such that $\lambda_i'=\lambda_i \xi$, and $\lambda'=\lambda \xi^d$. We then say $I$ and $I'$ are equivalent. By straightforward calculation, we have $S_I=S_{I'}$. Therefore, $S_I$ only depends on the equivalence class $[I]$ of $I$.

Suppose $H$ is a subgroup of $G$. Then the data $I$ induces a data $J=(H, \lambda\big{|}_H, \lambda_1\big{|}_{H}, \cdots, \lambda_N\big{|}_H)$. We then write $J\leqslant I$. If $H\subsetneq G$, then we can write $J<I$. Let $\calP=\calP_{d,N}$ be the set of equivalence classes of $I$. For $J<I$ (or $J\leqslant I$), we write $[J]<[I]$ (or $[J]\leqslant [I]$ respectively). In this way, $\calP$ is a poset. We call $I$ or $[I]$ maximal, if it is maximal with respect to this poset-structure.

In the rest of the paper, the symbol $S$ always refers to a set of monomials of degree $d$ in $x_1, \cdots, x_N$. As we have mentioned in the introduction, we define
\begin{equation*}
E_{S}\coloneqq\{(i,j)|x_i^{d-1}x_j\in S\}
\end{equation*}
to be a directed graph associated with $S$, where $[N]=\{1, \cdots, N\}$ is the set of vertices, and $(i,j)\in E_S$ is an arrow in the graph from $i$ to $j$.

\subsection{K,T,Y-type}
A homogeneous polynomial of degree $d$ is called of type $K$, if it is (up to a permutation of coordinates) of the form
 \begin{equation}
 \label{equation: K}
 x_1^{d-1} x_2+x_2^{d-1}x_3+\cdots+x_{k}^{d-1} x_1,
 \end{equation}
and is called of type $T$, if it is of the form
 \begin{equation}
 \label{equation: T}
 x_1^{d-1} x_2+\cdots+x_{k-1}^{d-1} x_k+x_k^d
 \end{equation}
For the case $d=3$, a polynomial of the form
\begin{equation}
\label{equation: Y}
x_1^2 x_2+\cdots+x_{a-1}^2 x_a+ y_1^2 y_2+ \cdots+y_{b-1}^2 y_b+  x_a^2 z+ y_b^2 z+ z^2 w+w^3+ x_a y_b w
\end{equation}
is called of type $Y$. Here, the types $K$ and $T$ are first introduced and used in \cite{zheng2020abelian}.

\begin{defn}
Let $N\geqslant 1$, a homogeneous polynomial $F(x_1, \cdots, x_N)$ is called smooth if the equations $\frac{\partial F}{\partial x_1}=\cdots=\frac{\partial F}{\partial x_N}=0$ have only one solution $(x_1, \cdots, x_N)=(0,\cdots, 0)$.
\end{defn}

It is straightforward to check that polynomials of type $K$ and $T$ are smooth. We will show (Proposition \ref{proposition: Y smooth}) that all $Y$-type polynomials are smooth.

\begin{defn}
\label{definition: KTY}
For $d\neq3$, a smooth homogeneous polynomial is called simple if it equals to the sum of finitely many polynomials of type $K$ or $T$ with independent variables. For $d=3$, a smooth homogeneous polynomial is called simple if it equals to the sum of finitely many polynomials of type $K,T$ or $Y$ with independent variables.
\end{defn}

\begin{defn}
\label{definition: simple large}
Let $S$ be a set of monomials of degree $d$. We call $S$ large if a generic linear combination of the elements in $S$ is a smooth polynomial, otherwise $S$ is called small. We call $S$ simple if the sum of certain elements in $S$ is a simple polynomial. We say $I=(G, \lambda, \lambda_1, \cdots, \lambda_N)$ is large, small, or simple if and only if $S_I$ is.
\end{defn}

The following lemma (which is equivalent to \cite[Lemma 3.2]{oguiso2019quintic}) will be useful.

\begin{lem}
\label{lemma: main criterion}
If $S$ is large, then for any distinct integers $a_1, \cdots, a_n, b_1, \cdots, b_k\in \{1,\cdots,N\}$ with $k<n$, there exists $m\in S$ such that $\sum_{j=1}^n \deg_{a_j}m\geqslant d-1$ and $\deg_{b_i}m=0$ for all $i\in\{1,\cdots,k\}$.
\end{lem}

We denote by $M(a_1,\cdots,a_n;b_1,\cdots,b_k)$ the set of monomials with $\sum_{i=j}^n \deg_{a_j}m\geqslant d-1$ and $\deg_{b_i}m=0$ for all $i\in\{1,\cdots,k\}$. This notation will be used frequently in the proof of Theorem \ref{theorem: main} and \ref{theorem: main2}.

\smallskip

A direct corollary of Lemma \ref{lemma: main criterion} is that, if $S$ is large, then for any $x_i$ there exists $x_j$ such that $(i, j)\in E_S$. Let $I=(G, \lambda, \lambda_1, \cdots, \lambda_N)$ be induced from a liftable action of $G$ on a smooth hypersurface $V(F)\subseteq\PP^{N-1}$ of degree $d$. Then the corresponding set $S_I$ contains $S(F)$ as a subset, hence is large. We consider chains of indices $(i_1, i_2, \cdots, i_s)$ such that
\begin{enumerate}[(a)]
\item $(i_t, i_{t+1})\in E_{S_I}$ for any $t\in\{1,\cdots, s-1\}$, and \label{rule chain}
\item $(i_s, i_{s-c})\in E_{S_I}$ for certain $c\in\{0,\cdots, s-1\}$, and \label{rule ring}
\item $\lambda_{i_1}, \cdots, \lambda_{i_s}$ are distinct. \label{rule lambda}
\end{enumerate}

Such chains clearly exist. We have the following lemma:
\begin{lem}
\label{lemma: restriction of edges in a spoon}
Given $I=(G, \lambda, \lambda_1, \cdots, \lambda_N)$ and a chain of indices $(i_1, i_2, \cdots, i_s)$ introduced above, then any $(i_a,i_b)\in E_{S_I}$ with $a,b\in[s]$ must be one of those proposed in (\ref{rule chain}) and (\ref{rule ring}).
\end{lem}
\begin{proof}
Assume to the contrary that $(i_a,i_b)\in E_{S_I}$ and it does not equal to any pair in (\ref{rule chain}) and (\ref{rule ring}). Let
\begin{equation*}
    \text{next}(a)=\begin{cases}
    a+1 & \text{if }a<s,\\
    c & \text{if }a=s,
    \end{cases}
\end{equation*}
then $(i_a,i_{\text{next}(a)})$ is the edge proposed in (\ref{rule chain}) or (\ref{rule ring}). Hence, $b\neq \text{next}(a)$. Since $x_{i_a}^{d-1}x_{i_b},$
$x_{i_a}^{d-1}x_{i_{\text{next}(a)}}\in S$, then $\lambda_{i_a}^{d-1}\lambda_{i_b}=\lambda_{i_a}^{d-1}\lambda_{i_{\text{next}(a)}}=\lambda$. This implies $\lambda_{i_b}=\lambda_{i_{\text{next}(a)}}$, which contradicts with (\ref{rule lambda}).
\end{proof}

\begin{defn}
\label{definition: s and c}
Let $s_I$ be the maximal possibility of $s$ among all choices, and let $c_I$ be the minimal possibility of $c$ among all choices such that $s=s_I$.
\end{defn}

Our proof of Theorem \ref{theorem: main} and \ref{theorem: main2} relies on a case-by-case analysis according to the value of $(s_I, c_I)$.

\smallskip

\begin{lem}
\label{lemma: any d}
Given $I=(G, \lambda, \lambda_1, \cdots, \lambda_N)$ which is large. If there exist $a,b,c,f\in[N]$ (which may equal to each other) such that $a\ne b$, $\{(a,c),(b,c),(c,f)\}\subseteq E_{S_I}$, $\lambda_c\ne \lambda_a, \lambda_b$, then there exists $e\in[N]\backslash\{c\}$ such that $(e,f)\in E_{S_I}$.
\end{lem}
\begin{proof}
From the assumption we know $a,b,c$ are distinct. By Lemma \ref{lemma: main criterion}, there exists a monomial $m\in S_I$ such that $\deg_a m+ \deg_b m\geqslant d-1$ and $\deg_c m=0$. We can write $m=x_a^{k}x_b^{d-k-1}x_e$ with $e\in[N]\backslash\{c\}$. We then have
\begin{equation*}
\lambda_a^k\lambda_b^{d-k-1}\lambda_e=\lambda_a^{d-1}\lambda_c=\lambda_b^{d-1}\lambda_c=\lambda,
\end{equation*}
which implies that
$$(\lambda_a^k\lambda_b^{d-k-1}\lambda_e)^{d-1}=(\lambda_a^{d-1}\lambda_c)^k(\lambda_b^{d-1}\lambda_c)^{d-k-1}.$$
After simplification, we have $\lambda_e^{d-1}=\lambda_c^{d-1}$, which implies that $(e,f)\in E_{S_I}$.
\end{proof}

The following proposition will be used in the proof of Theorem \ref{theorem: main} and \ref{theorem: main2} in the case $c_I>0$.
\begin{prop}
\label{proposition: loop}
Given $I=(G, \lambda, \lambda_1, \cdots, \lambda_N)$ which is large. If $s_I-1>c_I\geqslant 1$, then $N\geqslant s_I+c_I$.
\end{prop}
\begin{proof}
Without loss of generality, we let $(i,i+1)\in E_{S_I}$ for all positive integers $i<s_I$ and $(s_I,s_I-c_I)\in E_{S_I}$. We also assume that $\lambda_{1}, \cdots, \lambda_{s_{I}}$ are distinct. Firstly, using Lemma \ref{lemma: any d} with $(a,b,c,f)=(s_I-c_I-1,s_I,s_I-c_I,s_I-c_I+1)$, there exists $e\in[N]\backslash\{s_I-c_I\}$ such that $(e,s_I-c_I+1)\in E_{S_I}$. Then, according to Proposition \ref{lemma: restriction of edges in a spoon}, we must have $e>s_I$. Without loss of generality, we let $e=s_I+1$.

Suppose $c_I=1$, then we are done since $N\geqslant s_I+1$. Suppose $c_I\geqslant 2$. Using Lemma \ref{lemma: any d} with $(a,b,c,f)=(s_I-c_I, s_I+1, s_I-c_I+1, s_I-c_I+2)$, there exists $e'\in [N]\backslash\{s_I-c_I+1\}$ such that $(e', s_I-c_I+2)\in E_{S_I}$. By Proposition \ref{lemma: restriction of edges in a spoon}, we must have $e'>s_I$. We have $e'\ne e$ since $(e, s_I-c_I+1), (e', s_I-c_I+2)\in E_{S_I}$ and $\lambda_{s_I-c_I+1}\ne \lambda_{s_I-c_I+2}$. Therefore, we may assume (without loss of generality) $e'=s_I+2$.

Continuing the above argument, we may ask $(s_I+k, s_I-c_I+k)\in E_{S_I}$ with $k\in [c_I]$. In particular, $N\ge s_I+c_I$.
\end{proof}





\section{Cubic Hypersurfaces}
\label{section: cubic}
In this section we focus on cubic hypersurfaces. We give a useful lemma in \S\ref{subsection: key lemma}, and discuss the properties of polynomials of type $Y$ in \S \ref{Y-type}.

\subsection{A useful lemma}
\label{subsection: key lemma}
The next lemma will be repeatedly used in the proof of Theorem \ref{theorem: main}.

\begin{lem}
\label{lemma: key}
Let $d=3$. Given $I=(G, \lambda, \lambda_1, \cdots, \lambda_N)$ which is large. If there exist $a,b,c,f\in [N]$ (which may equal to each other) such that $a\ne b$, $\{(a,c),(b,c),(c,f)\}\subseteq E(\coloneqq E_{S_I})$ and $\lambda_c\ne\lambda_a, \lambda_b$, then there exists $e\in\{1,\cdots, N\}\backslash\{a,b,c\}$ such that $(e,f)\in E$, and either $\lambda_e=\lambda_c$ or $x_ax_bx_e\in S_I$.
\end{lem}
\begin{proof}
We have $a,b,c$ distinct. By Lemma \ref{lemma: main criterion}, there exists a monomial $m$ with $\deg_a m+ \deg_b m\geqslant 2$ and $\deg_c m=0$. Suppose $\deg_a m+ \deg_b m=3$. If $\deg_a m=3$, then $\lambda_a^3=\lambda_a^2\lambda_c=\lambda$, which implies that $\lambda_a=\lambda_c$, a contradiction. The case $\deg_b m=3$ is similar. If $\deg_a m=2, \deg_b m=1$, then $\lambda_a^2\lambda_b=\lambda_a^2\lambda_c=\lambda$, which implies that $\lambda_b=\lambda_c$, again a contradiction (the case $\deg_b m=2$ is similar). Therefore, we must have $\deg_a m+ \deg_b m=2$. In particular, there exists $e\in[N]\backslash\{a,b,c\}$, such that $\deg_e m=1$.

If $\deg_a m=2$, then $m=x_a^2 x_e$. As both $x_a^2 x_c, x_a^2 x_e$ belong to $S$, we have $\lambda_e=\lambda_c$. Combining with the fact $x_c^2 x_f\in S$, we conclude that $x_e^2 x_f\in S$. The case $\deg_b m=2$ is similar.

If $\deg_a m=\deg_b m=1$, then $m=x_ax_bx_e$. From $$\lambda_a^2\lambda_c=\lambda_b^2\lambda_c=\lambda_a\lambda_b\lambda_e=\lambda$$ we have
$$\lambda_a^2 \lambda_b^2\lambda_c^2=\lambda_a^2\lambda_b^2\lambda_e^2=\lambda^2,$$
which implies that $\lambda_c^2=\lambda_e^2$. Thus $(e,f)\in E$.
\end{proof}

\subsection{Y-type}
\label{Y-type}
Recall that a cubic polynomial of the form \eqref{equation: Y} is called of type $Y$. We next show the smoothness of a $Y$-type polynomial.

\begin{prop}
\label{proposition: Y smooth}
A $Y$-type cubic polynomial is smooth.
\end{prop}
\begin{proof}
We use the expression \eqref{equation: Y} for a $Y$-type polynomial. We prove the smoothness by induction on $(a,b)$ for $a,b\geqslant 1$. If $(a,b)=(1,1)$, we let $F=x^2z+y^2z+z^2w+w^3+xyw$ be the corresponding $Y$-type polynomial, where we simply replace $x_1, y_1$ by $x, y$. We have:
\begin{numcases}{}
    \frac{\partial F}{\partial x}=2xz+yw=0 \label{partial x} \\
    \frac{\partial F}{\partial y}=2yz+xw=0 \label{partial y}\\
    \frac{\partial F}{\partial z}=x^2+y^2+2zw=0 \label{partial z}\\
    \frac{\partial F}{\partial w}=z^2+3w^2+xy=0.\label{partial w}
\end{numcases}

Suppose to the contrary that $(x,y,z,w)\neq (0,0,0,0)$ is a solution to the set of equations.
Suppose $z=w=0$, then from \eqref{partial z},\eqref{partial w} we obtain $x^2+y^2=xy=0$, which implies that $x=y=0$, contradiction.
Suppose $z\neq 0$ and $w=0$, then from \eqref{partial x},\eqref{partial y} we obtain $x=y=0$. Then \eqref{partial w} cannot hold, contradiction.
Hence we have $w\neq 0$. We assume (without loss of generality) $w=1$.

Now, from \eqref{partial x} we have $y=-2xz$. Combining with \eqref{partial y}, we obtain $-4xz^2+x=0$, hence either $x=0$ or $z^2=\frac{1}{4}$. If $x=0$, then by \eqref{partial x} we have $y=0$. By \eqref{partial z} we have $z=0$. Then the equation \eqref{partial w} cannot hold, contradiction. If $z^2=\frac{1}{4}$, then by \eqref{partial w} we have $xy=-\frac{13}{4}$. If $z=\frac{1}{2}$, then $y=-x$ and $x^2=y^2=-xy=\frac{13}{4}$, which contradicts with \eqref{partial z}. The case $z=-\frac{1}{2}$ is similar. We conclude the smoothness of $F$ for $(a,b)=(1,1)$.

We then prove the smoothness for $(a,b)=(2,1)$. Let $F=x_1^2x_2+x_2^2z+y_1^2z+z^2w+w^3+x_2y_1w$ be the corresponding polynomial. We consider the following $4$ partial derivatives:
\begin{numcases}{}
    \frac{\partial F}{\partial x_1}=2x_1x_2=0 \label{2partial x_1} \\
    \frac{\partial F}{\partial x_2}=x_1^2+2x_2z+y_1w=0 \label{2partial x_2} \\
    \frac{\partial F}{\partial y_1}=2y_1z+x_2w=0 \label{2partial y_1}\\
    \frac{\partial F}{\partial z}=x_2^2+y_1^2+2zw=0 \label{2partial z}
\end{numcases}
From \eqref{2partial x_1} we know that $x_1=0$ or $x_2=0$. If $x_1=0$, then the situation is reduced to the case $(a,b)=(1,1)$. Thus we must have $(x_2,y_2,z,w)=(0,0,0,0)$. Next we assume $x_1\neq0$ and $x_2=0$. By \eqref{2partial x_2} and \eqref{2partial y_1} we have $y_1w\neq 0$ and $y_1z=0$, hence we have $y_1\ne 0$ and $z=0$. Then the equation \eqref{2partial z} cannot hold. We conclude the smoothness of $F$ for $(a,b)=(2,1)$.

Next we consider the case $(a,b)=(2,2)$. Let $F=x_1^2x_2+x_2^2z+y_1^2y_2+y_2^2z+z^2w+w^3+x_2y_2w$ be the corresponding polynomial. We consider
\begin{numcases}{}
    \frac{\partial F}{\partial x_1}=2x_1x_2=0 \label{3partial x_1} \\
    \frac{\partial F}{\partial x_2}=x_1^2+2x_2z+y_2w=0 \label{3partial x_2} \\
    \frac{\partial F}{\partial y_1}=2y_1y_2=0 \label{3partial y_1}
\end{numcases}

If $x_1=0$, then we can reduce the situation to the case $(a,b)=(2,1)$. Next we assume $x_1\neq 0$ and $x_2=0$. Symmetrically, we assume moreover $y_1\neq 0, y_2=0$. Then the equation \eqref{3partial x_2} cannot hold, a contradiction. The case $(a,b)=(2,2)$ is done.

Now, we let $F_{a,b}$ be the equation  \ref{equation: Y}. We assume the smoothness of $F_{a,b}$  for $(a,b)$ with $a\geqslant 2,b\geqslant 1$. We next show that $F_{a+1,b}$ is smooth. Firstly, if $x_1=0$, then we can reduce the situation to the case of $(a,b)$. Then we assume $x_1\neq 0$ and $x_2=0$. But in such a situation the partial derivative
$$\frac{\partial F_{a+1,b}}{\partial x_2}=x_1^2+2x_2x_3$$
cannot vanish. This completes the proof of the proposition.
\end{proof}

The next proposition characterizes the groups of diagonal automorphisms of polynomials of type $Y$.
\begin{prop}
For the $Y$-type polynomial
\begin{equation*}
F=x_1^2 x_2+\cdots+x_{a-1}^2 x_a+ y_1^2 y_2+ \cdots+y_{b-1}^2 y_b+  x_a^2 z+ y_b^2 z+ z^2 w+w^3+ x_a y_b w
\end{equation*}
defined in \eqref{equation: Y}. Suppose $a\geqslant b$, then the group $G_F$ of diagonal automorphisms of $V(F)$ is generated by two elements $g_1, g_2$ which can be described as follows. A diagonal element of $V(F)$ is denoted by $\diag[\lambda_{x_1}:\cdots:\lambda_{x_a}:\lambda_{y_1}:\cdots:\lambda_{y_b}:z:w]$. The element $g_1$ is determined by $\lambda_{x_1}=\exp(\frac{(-1)^{a-1}\pi \sqrt{-1}}{2^a})$, $\lambda_{y_1}=\exp(\frac{(-1)^b\pi \sqrt{-1}}{2^b})$ and $\lambda_w=1$; another element $g_2$ is determined by $\lambda_{x_1}=1$, $\lambda_{y_1}=\exp(\frac{(-1)^{b-1}\pi \sqrt{-1}}{2^{b-2}})$ and $\lambda_w=1$. In particular $G_F\cong C_{2^{a+1}}\oplus C_{2^{b-1}}$.
\end{prop}
\begin{proof}
From $a\geqslant b$, we have $\langle g_1\rangle\cap\langle g_2\rangle=\{\diag[1:\cdots:1]\}$. Thus $\langle g_1, g_2\rangle\cong C_{2^{a+1}}\oplus C_{2^{b-1}}$.

Suppose $g=\diag[\lambda_{x_1}(g): \cdots: \lambda_w(g)]$ is a diagonal automorphism of $V(F)$. Without loss of generality, we let $\lambda_w(g)=1$. Then we have $\lambda(g)=\lambda_w(g)^3=1$.


We have $\lambda_{x_i}(g)=\lambda_{x_1}(g)^{(-2)^{i-1}}$ for all positive integers $i\leqslant a$. Since $x_a^2 z, z^2 w\in S(F)$, we have $\lambda_{x_1}(g)^{(-2)^{a+1}}=\lambda_w(g)=1$. Since $\lambda_{x_1}(g_1)=\exp(\frac{(-1)^{a-1}\pi \sqrt{-1}}{2^a})$ is a primitive $2^{a+1}$-root of unity, there exists a positive integer $k$ such that $\lambda_{x_1}(g)=\lambda_{x_1}(g_1)^k$. Let $g'\coloneqq g/g_1^k$. Then we can write $g'=\diag[\lambda_{x_1}(g'):\cdots:\lambda_w(g')]$ such that $\lambda_{x_i}(g')=\lambda_w(g')=1$. Combining with $x_a y_b w\in S$, we have $\lambda_{x_a}(g')=\lambda_{y_b}(g')=1$. It is then straightforward to show that $g'$ is a power of $g_2$. We conclude that $g\in \langle g_1, g_2\rangle$.
\end{proof}


\section{Cubic Fourfolds}
\label{section: cubic fourfold}
\subsection{Main Theorem}
In this section we concentrate on the case of cubic fourfolds. The main result is Theorem \ref{theorem: main}, where we essentially classify all liftable abelian automorphism groups of smooth cubic fourfolds by finding all non-simple and maximal cases.

Recall that in Definition \ref{definition: KTY} we have defined what are simple polynomials. By Definition \ref{definition: simple large}, a set $S$ of monomials is called simple, if $S$ contains the monomials of a simple cubic polynomial in $x_1, \cdots, x_N$.
The terminology \underline{simple} used here is slightly different from that is used in \cite{zheng2020abelian}. In \cite{zheng2020abelian} we did not consider polynomials of type $Y$. The following lemma is clear.





\begin{lem}
\label{lemma: d=3 c>1}
For $d=3$ and $I=(G,S)$, we have either $c_I=0$ or $c_I\geqslant 2$.
\end{lem}

Recall from \S\ref{subsection: poset} that $\calP_{d,N}$ is the set of equivalence classes of $I=(G, \lambda, \lambda_1, \cdots, \lambda_N)$, and it is naturally a poset. The next theorem describes the maximal elements in $\calP_{3,6}$.

\begin{thm}
\label{theorem: main}
Suppose $G$ is an abelian group with a faithful and liftable action on a smooth cubic fourfold $X$. Choose a coordinate system $x_1, \cdots, x_6$ for $\PP^5$ such that the action of $G$ is diagonal. Let $I=(G, \lambda, \lambda_1, \cdots, \lambda_6)$ be the associated data. Suppose $[I]$ is maximal in $\calP_{3,6}$, then either $S=S_I$ is simple, or we have the following possibilities (up to a permutation of coordinates $x_1, \cdots, x_6$):
\begin{enumerate}[(1)]
    \item $G=\{\diag[\omega:\omega^5:\omega^3:1:\omega^4:\omega^2]|\omega^6=1\}\cong C_2\oplus C_3,\\S=\{x_1^2x_2,x_2^2x_3,x_3^2x_1,x_4^2x_1,x_5^2x_2,x_6^2x_3,x_3x_4x_5,x_1x_5x_6,x_2x_4x_6\},$
    \item $G=\{\diag[\omega:\omega^6:\omega^4:1:\omega^5:\omega^2]|\omega^8=1\}\cong C_8,\\S=\{x_1^2x_2,x_2^2x_3,x_3^2x_4,x_4^3,x_5^2x_2,x_1x_5x_6,x_6^2x_3,x_2x_6x_4\},$
    \item $G=\{\diag[\omega:\omega^2:1:\zeta:\omega^2\zeta:1]|\omega^4=\zeta^2=1\}\cong C_2\oplus C_4,\\S=\{x_1^2x_2,x_2^2x_3,x_2^2x_6,x_2x_4x_5,x_3^3,x_3^2x_6,x_4^2x_3,x_5^2x_3,x_6^2x_3,x_4^2x_6,x_5^2x_6,x_6^3\},$
    \item $G=\{\diag[\omega:\omega^2:1:\omega\zeta:\omega^2\zeta:1]|\omega^4=\zeta^2=1\}\cong C_2\oplus C_4,\\S=\{x_1^2x_2,x_2^2x_3,x_3^3,x_4^2x_2,x_1x_4x_5,x_5^2x_3,x_5^2x_6,x_6^2x_3,x_2^2x_6,x_3^2x_6,x_6^3\}$
    \item $G=\{\diag[\omega:\omega^2:1:\omega\zeta:\omega^2\zeta:\zeta]|\omega^4=\zeta^2=1\}\cong C_2\oplus C_4,\\S=\{x_1^2x_2,x_2^2x_3,x_3^3,x_4^2x_2,x_1x_4x_5,x_5^2x_3,x_2x_5x_6,x_6^2x_3\},$
    \item $G=\{\diag[\omega:\omega^2:1:\omega^2\zeta:\zeta:\omega^3\zeta]|\omega^4=\zeta^2=1\}\cong C_2\oplus C_4,\\S=\{x_1^2x_2,x_2^2x_3,x_3^3,x_4^2x_3,x_5^2x_3,x_2x_4x_5,x_1x_5x_6,x_6^2x_2\},$
    \item $G=\{\diag[\alpha:\beta:\alpha\beta:1:\omega:\zeta]|\alpha^2=\beta^2=\omega^3=\zeta^3=1\}\cong C_2^2\oplus C_3^2,\\S=\{x_1^2x_4,x_2^2x_4,x_3^2x_4,x_1x_2x_3,x_4^3,x_5^3,x_6^3\},$
    \item $G=\{\diag[\alpha:\beta:\alpha\beta:\omega:1:\omega^4]|\alpha^2=\beta^2=\omega^6=1\}\cong C_2^3\oplus C_3,\\S=\{x_1^2x_5,x_2^2x_5,x_3^2x_5,x_1x_2x_3,x_5^3,x_4^2x_6,x_6^3\}.$
\end{enumerate}
Here $C_n\coloneqq \ZZ/n\ZZ$ is the cyclic group of order $n$.
\end{thm}

\begin{proof}
We have defined two invariants $s=s_I, c=c_I$ for $I$, see the end of \S \ref{section: pre}. Suppose $c\ne 0$, by Lemma \ref{lemma: d=3 c>1} we have $c \geqslant 2$. If $s>c+1$, then by Proposition \ref{proposition: loop}, we have $N\geqslant s+c\geqslant 2(c+1)$. Then we must have $c=2$ and $s=4$. With a suitable choice of coordinates $x_1, \cdots, x_6$, we may ask $S\supset \{x_1^2 x_2, x_2^2 x_3, x_3^2 x_1, x_4^2 x_1, x_5^2 x_2, x_6^2 x_3\}$. By Lemma \ref{lemma: key}, we have either $\lambda_5=\lambda_1$ or $x_3 x_4 x_5\in S$. If $\lambda_5=\lambda_1$, we claim that we must have $\lambda_6=\lambda_2$. In fact, using Lemma \ref{lemma: key} again, we have either $\lambda_6=\lambda_2$ or $x_1 x_5 x_6\in S$. If the latter happens, then $\lambda_1\lambda_5\lambda_6=\lambda_1^2 \lambda_2$, which still implies that $\lambda_6=\lambda_2$.

Similarly, each of the three equalities $\lambda_5=\lambda_1, \lambda_6=\lambda_2, \lambda_4=\lambda_3$ implies the other two. If all the equalities hold, then $S$ contains $\{x_1^2 x_2, x_2^2 x_3, x_3^2 x_1, x_4^2 x_5, x_5^2 x_6, x_6^2 x_4\}$, hence simple. If not, then $S$ contains the set
\begin{equation*}
\{x_1^2x_2,x_2^2x_3,x_3^2x_1,x_4^2x_1,x_5^2x_2,x_6^2 x_3,x_3 x_4 x_5,x_1 x_5 x_6,x_2 x_4 x_6\}
\end{equation*}
An element in $G$ could be represented as $\diag[\omega:\lambda_2:\lambda_3:1:\lambda_5:\lambda_6]$. Then
\begin{equation*}
\omega^2 \lambda_2=\lambda_2^2\lambda_3=\lambda_3^2 \omega=\omega=\lambda_3 \lambda_4 \lambda_5=\lambda_1 \lambda_5 \lambda_6,
\end{equation*}
hence $\lambda_2=\omega^{-1}, \lambda_3=\omega^3, \lambda_5=\omega^4, \lambda_6=\omega^2$ and $\omega^6=1$. This gives rise to case (1).

If $c=2$ and $s=c+1$, we may ask $S\supset \{x_1^2 x_2, x_2^2 x_3, x_3^2 x_1\}$ and assume moreover that $\lambda_1, \lambda_2, \lambda_3$ are distinct. There exists $i_4, i_5, i_6\in [6]$ such that $(4, i_4), (5, i_5), (6, i_6)\in E$. If one of $i_4, i_5, i_6$ can be chosen from $[3]$, we can assume (without loss of generality) that $i_4=2$. From the maximality of $s=3$, we know that $\lambda_4$ is equal to one of $\lambda_1, \lambda_2, \lambda_3$. The only possibility is $\lambda_4=\lambda_1$. By Lemma \ref{lemma: main criterion}, there exists $j\in[6]\backslash\{1,2,4\}$ such that $(4, j) \in E$. We may assume $\lambda_5=\lambda_2$. We can furthermore assume that $\lambda_6=\lambda_3$. Thus $S$ is simple.

Suppose $i_4, i_5, i_6$ all belong to $\{4,5,6\}$. If $(4,4), (5,5), (6,6)\in E$, then $S$ is simple. We assume $(5,6)\in E$. If $i_6=4$, then whatever $i_4$ is, we can conclude that $S$ is simple. If $i_6=5$, then $\lambda_5=\lambda_6$, which implies that $(5,5), (6,6)\in E$. Whatever $i_4$ is, we have that $S$ is simple. Suppose $i_6=6$. If $i_4=4$ or $5$, then $S$ is simple. We reduce to the case $(4,6), (5,6), (6,6)\in E$. By Lemma \ref{lemma: main criterion}, there exists $m\in M(4,5;6)$. If $\deg_4 m+\deg_5 m=3$, then $S$ is simple; otherwise, we may assume $\deg_1 m=1$. Then $\lambda_1^2=\lambda_6^2$, which implies that $(6,2)\in E$, which is already excluded. A similar argument in this paragraph will be used in the proof of case $c=0, s=3$ and $i_4, i_5, i_6\in \{4,5,6\}$.

If $c\ge 3$ and $s=c+1$, a similar argument as above shows that $S$ is simple.

\vspace{3ex}

Next we assume $c=0$. If $s=6$ or $1$, then $S$ is simple. Next we assume $2\leqslant s\leqslant 5$.

\smallskip

If $s=5$, assume that $\{x_1^2 x_2,x_2^2 x_3,x_3^2x_4,x_4^2x_5,x_5^3\}\subseteq S$, then we have $x_6^2 x_{i_6}\in S$ for certain $i_6\in\{1,2,3,4,5\}$. If $i_6=1$, then $S$ is simple.

If $i_6=2$, then $x_6^2 x_2\in S$. We have $\lambda_6^2 \lambda_2=\lambda_2^2\lambda_3$ and $\lambda_2\ne\lambda_3$, which implies that $\lambda_6\ne \lambda_2$. Thus $\lambda_2\notin\{\lambda_1, \lambda_6\}$. Then by Lemma \ref{lemma: key} with $(a,b,c,f)=(1,6,2,3)$, there exists $e\in\{3,4,5\}$ such that $(e,3)\in E$. We have also $(e,\min(e+1,5))\in E$. Thus $\lambda_3=\lambda_{\min(e+1,5)}$, which contradicts the definition of $s_I$. Similarly, if $i_6=3$, we take $(a,b,c,f)=(2,6,3,4)$ in Lemma \ref{lemma: key}, then there exists $e\in\{1,4,5\}$ with $(e,4)\in E$, which implies that $\lambda_4=\lambda_{\min(e+1,5)}$, again a contradiction.

If $i_6=4$, then by Lemma \ref{lemma: key} with $(a,b,c,f)=(3,6,4,5)$, there exists $e\in\{1,2,5\}$ such that $(e,5)\in E$. The only possibility is $e=5$, then $x_3x_6x_5\in S$. Now, $S=\{x_1^2x_2,x_2^2x_3,x_3^2x_4,x_4^2x_5,x_5^3,x_6^2x_4,x_3x_6x_5\}$ is large. This is of type $Y$.

If $i_6=5$, then by Lemma \ref{lemma: key} with $(a,b,c,f)=(4,6,5,5)$, there exists $e\in \{1,2,3\}$ such that $(e,5)\in E$, which is impossible.

\smallskip

If $s=4$, then we suppose $\{x_1^2x_2,x_2^2x_3,x_3^2x_4,x_4^3\}\subseteq S$. Take $i_5, i_6$ such that $(5,i_5),(6,i_6)\in E$. Without loss of generality, we assume $i_5\leqslant i_6$. By the minimality of $s$, we must have $i_5\geqslant 2$.

Suppose $i_5=2$, then using Lemma \ref{lemma: key} with $(a,b,c,f)=(1,5,2,3)$, there exists $e\in\{3,4,6\}$ such that $(e,3)\in E$. The only possibility is $e=6$, hence we may ask $i_6=3$. Now using Lemma \ref{lemma: key} with $(a,b,c,f)=(2,6,3,4)$, there exists $e_2\in\{1,5,4\}$ such that $(e_2,4)\in E$. The only possibility is $e_2=4$. Since $\lambda_4\neq\lambda_3$, it holds $x_2x_6x_4\in S$. By definition of $e$, we have either $\lambda_6=\lambda_2$ or $x_1x_5x_6\in S$. If $\lambda_6=\lambda_2$, then $x_5^2x_6\in S$, which implies that $S$ is of type $Y$ and hence simple. If $x_1x_5x_6\in S$, then
$S=\{x_1^2x_2,x_2^2x_3,x_3^2x_4,x_4^3,x_5^2x_2,x_1x_5x_6,x_6^2x_3,x_2x_6x_4\}$
is large and $G=\{\diag[\omega:\omega^6:\omega^4:1:\omega^5:\omega^2]|\omega^8=1\}$. This is case (2).

Suppose $i_5=3$, then by Lemma \ref{lemma: key} with $(a,b,c,f)=(2,5,3,4)$, there exists $e\in\{1,4,6\}$ such that $(e,4)\in E$. The only possibilities for $e$ are $4$ and $6$.

If $e=6$, then we may ask $i_6=4$. By Lemma \ref{lemma: key} with $(a,b,c,f)=(3,6,4,4)$, there exists $e_2\in\{1,2,5\}$ such that $(e_2,4)\in E$, where we have no choice.

If $e=4$, then $x_2x_5x_4\in S$. If $i_6\in\{5,6\}$, then $S$ is simple. The possibility $i_6=4$ was already excluded before. If $i_6=3$, then by Lemma \ref{lemma: key} with $(a,b,c,f)=(2,6,3,4)$, there exists $e_2\in \{1,4,5\}$ such that $(e_2,4)\in E$. The only possibility is $e_2=4$. Since $\lambda_3\neq \lambda_4$, we must have $x_2x_6x_4\in S$. Using Lemma \ref{lemma: key} again for $(a,b,c,f)=(5,6,3,4)$, there exists $e_3\in\{1,2,4\}$ with $(e_3, 4)\in E$. The only possibility is $e_3=4$. Then we have $x_5 x_6 x_4\in S$. As  $x_2x_6x_4, x_2x_5x_4, x_5 x_6 x_4$ all belong to $S$, it holds $\lambda_2=\lambda_5=\lambda_6$. Thus $\lambda_2^2\lambda_4=\lambda_2\lambda_5\lambda_4=\lambda=\lambda_2^2\lambda_3$, which implies that $\lambda_3=\lambda_4$. But this contradicts our assumption.

Suppose $i_5=4$, then by Lemma \ref{lemma: key} with $(a,b,c,f)=(3,5,4,4)$, there exists $e\in\{1,2,6\}$ such that $(e,4)\in E$. The only possibility is $e=6$. Thus we may ask $i_6=4$.

If $\lambda_6=\lambda_4$, then $x_5^2x_6,x_6^3\in S$, which implies that $S$ is simple. Suppose $\lambda_6\ne\lambda_4$, then $x_3x_5x_6\in S$. An element in $G$ could be represnted as $\diag[\omega:\omega^6:\omega^4:1:\zeta\omega^4:\zeta]$ with $\omega^8=\zeta^2=1$. Since $\lambda_3\ne \lambda_4$, we can choose this element with $\omega^4\ne 1$. By Lemma \ref{lemma: main criterion}, there exists $m\in S$ such that $m\in M(2,5,6;3,4)$, which gives rise to a constraint in the form $\omega^a\zeta^b=1$ with nonnegative integers $a<8,b<2$. From $\zeta^2=1$ we have $\omega^{2a}=1$.

Suppose $\deg_1m=1$, then $a$ is odd. From $\omega^8=\omega^{2a}=1$ we deduce that $\omega^2=1$, which is a contradiction. Hence $\deg_1 m=0$. Then we have $\deg_2 m+\deg_5 m+\deg_6 m=3$. If $b=0$, then $\deg_5 m+\deg_6 m$ is even and $\deg_2 m$ is odd, which implies that $a\equiv 2$ (mod 4). From $\omega^{4a}=\omega^8=1$ we conclude $\omega^4=1$, which is a contradiction.

If $b=1$, then $\deg_5 m+\deg_6 m$ is odd and $\deg_2 m$ is even. Hence $a\equiv 0$ (mod 4). Then we have either $\zeta=1$ or $\zeta=\omega^4$. If $\zeta=1$, then $x_5^2x_6, x_6^3\in S$, hence $S$ is simple. If $\zeta=\omega^4$, then $x_5^3, x_6^2 x_5\in S$, hence $S$ is again simple. We have finished the discussion of case $s=4$.

\smallskip

If $s=3$, then there exist $(4,i_4),(5,i_5),(6,i_6)\in E,i_4,i_5,i_6\in [6]$. Suppose $i_4,i_5,i_6\in\{4,5,6\}$, then from the argument in the proof of case $c=2, s=3$, we only need to treat the case $(4,6), (5,6), (6,6)\in E$, and $\lambda_4\ne \lambda_6$, $\lambda_5\ne \lambda_6$. By Lemma \ref{lemma: key} with $(a,b,c,f)=(4,5,6,6)$, there exists $e\in\{1,2,3\}$ such that $(e,6)\in E$. If $e=1$, then $\lambda_2=\lambda_6$. Since $x_6^3\in S$, then $x_2^3\in S$, which implies $\lambda_2=\lambda_3$, a contradiction.


Next, whenever $e=2$ or $3$, we always have $\lambda_6=\lambda_3$. By Lemma \ref{lemma: main criterion}, there exists a monomial $m\in S$ such that $m\in M(2,4,5; 3,6)$. Suppose $\deg_1m=0$, then $m=x_2x_4x_5$ is the only choice. In this case we have a large and non-simple $S=\{x_1^2x_2,x_2^2x_3,x_2^2x_6,x_2x_4x_5,x_3^3,x_3^2x_6,$
$x_4^2x_3,x_5^2x_3,x_6^2x_3,x_4^2x_6,x_5^2x_6,x_6^3\}$, and $G=\{\diag[\omega:\omega^2:1:\zeta:\omega^2\zeta:1]|\omega^4=\zeta^2=1\}$. In particular, a smooth polynomial is $F=x_1^2x_2+x_2^2x_3+x_3^3+x_4^2x_6+2x_5^2x_6+3x_6^3+x_2x_4x_5+x_4^2x_3+x_5^2x_3$. This is case (3).

Next, we assume the existence of $u\in\{4,5,6\}$, such that $(u,4), (u,5), (u,6)\notin E$. Without loss of generality, we assume that $u=4$, and $(u,v)\in E$ for $v=1,2$ or $3$. Since $(4,4)\notin E$, we have $\lambda_4\ne \lambda_3$. From the maximality of $s$, we cannot have $v=1$. Hence we have either $(4,2)\in E$ or $(4,3)\in E$.

Suppose $(4,2)\in E$, then by Lemma \ref{lemma: key} with $(a,b,c,f)=(1,4,2,3)$, there exists $e\in\{3,5,6\}$ such that $(e,3)\in E$. If $e\neq 3$, then we assume (without loss of generality) $e=5$, so $(5,3)\in E$. By Lemma \ref{lemma: key} with $(a,b,c,f)$=(2,5,3,3), there exists $e_2\in\{1,4,6\}$ such that $(e_2,3)\in E$. The only choice is $e_2=6$, so $(6,3)\in E$.

From the definition of $e_2$, we have to choose either $\lambda_6=\lambda_3$ or $x_2x_5x_6\in S$. Suppose $\lambda_6=\lambda_3$, then $(5,6),(6,6)\in E$. Next, from the definition of $e$, we have to choose either $\lambda_5=\lambda_2$ or $x_1x_4x_5\in S$. If $\lambda_5=\lambda_2$, then $(4,5)\in E$, which makes $S$ simple. If $x_1x_4x_5\in S$, then $S=\{x_1^2x_2,x_2^2x_3,x_3^3,x_4^2x_2,x_1x_4x_5,x_5^2x_3,x_5^2x_6,x_6^2x_3,x_2^2x_6,x_3^2x_6,x_6^3\}$
is large, and $G=\{\diag[\omega:\omega^2:1:\omega\zeta:\omega^2\zeta:1]|\omega^4=\zeta^2=1\}$. This is case (4).

Suppose $x_2x_5x_6\in S$ and $\lambda_6\neq \lambda_3$. If $\lambda_2=\lambda_5$, then $x_2^2x_6\in S$, which implies that $\lambda_6=\lambda_3$ and this is a contradiction. If $x_1x_4x_5\in S$, then $S=\{x_1^2x_2,x_2^2x_3,x_3^3,x_4^2x_2,x_1x_4x_5,x_5^2x_3,$
$x_2x_5x_6,x_6^2x_3\}$
is large, and $G=\{\diag[\omega:\omega^2:1:\omega\zeta:\omega^2\zeta:\zeta]|\omega^4=\zeta^2=1\}$. This is case (5).

If $e=3$, since $\lambda_3\neq\lambda_2$, we must have $x_1x_4x_3\in S$. Suppose $i_5,i_6\in\{5,6\}$, then $S$ is simple because the monomials in $\{x_1^2 x_2, x_2^2 x_3, x_3^3, x_4^2 x_2, x_1 x_3 x_4\}$ form a polynomial of $Y$-type. We suppose $i_5=2$, then according to Lemma \ref{lemma: key} with $(a,b,c,f)=(1,5,2,3)$, there exists $e_2\in\{3,4,6\}$ such that $(e_2,3)\in E$. If $e_2=4$, then $\lambda_2=\lambda_3$, which is not the case.

If $e_2=6$, then $(6,3)\in E$. According to Lemma \ref{lemma: key} with $(a,b,c,f)=(2,6,3,3)$, we have $e_3\in\{1,4,5\}$ such that $(e_3,3)\in E$. However, whatever it is, we have $\lambda_2=\lambda_3$, which is not the case.

If $e_2=3$, then $x_1x_5x_3\in S$, which implies that $\lambda_4=\lambda_5$ because $x_1x_4x_3\in S$. According to Lemma \ref{lemma: main criterion}, there exists a monomial $m\in S$ such that $m\in M(4,5;2)$. Since $\lambda_4=\lambda_5$, we may ask $\deg_5 m=0$. The only possibility is $m=x_4^2x_6$, then $\lambda_2=\lambda_6$. In this case, we have $(6,3)\in E$, which is already excluded on the last paragraph. We are done with the case $(4,2)\in E$.

Assume $(4,3)\in E$. According to Lemma $\ref{lemma: key}$ with $(a,b,c,f)=(2,4,3,3)$, we can suppose (without loss of generality) $e=5$ and $(5,3)\in E$. We have either $\lambda_5=\lambda_3$ or $x_2 x_4 x_5\in S$. Suppose $\lambda_5=\lambda_3$, then $(4,5)\in E$, which is a contradiction because we have taken $u=4$ such that $(u, 4), (u, 5), (u, 6) \notin E$.

If $\lambda_5\ne \lambda_3$, then we have $x_2 x_4 x_5\in S$.  If $\lambda_6=\lambda_3$, then $(4,6)\in E$, which is a contradiction. Since $(3,3)\in E$ and $(4,4)\notin E$, we have $\lambda_3\ne \lambda_4$. Next we have $\lambda_3\notin\{\lambda_4, \lambda_5, \lambda_6\}$. From $(5,3)\in E$ we have $(5,4), (5,5), (5,6)\notin E$.

According to Lemma \ref{lemma: main criterion}, we have a monomial $m\in S$ such that $m\in M(1,4,5;2,3)$. Since $x_4,x_5$ are symmetric, we suppose $\deg_4 m\leqslant\deg_5 m$.

Fristly, we suppose $\deg_6 m=0$. Since $x_2x_4x_5\in S$ and $\lambda_1\neq \lambda_2$, then $m\neq x_1x_4x_5$. Since $\lambda_1,\lambda_4,\lambda_5\neq\lambda_3$, then $m\notin\{x_5^2x_1,x_5^2x_4,x_5^3\}$. Since $x_2x_4x_5,x_2^2x_3\in S$ and $\lambda_3\neq\lambda_4$, we have $\lambda_2\neq\lambda_5$, which implies $m\neq x_1^2x_5$. We excluded all possibilities, thus a contradiction.

Next we suppose $\deg_6m=1$. Since $\lambda_6\neq\lambda_3$, then $m\neq x_5^2x_6$. The two choices $m\in\{x_1^2x_6,x_4x_5x_6\}$ are equivalent because they are both equivalent to $\lambda_6=\lambda_2$. Suppose this is the case, then $x_6^2x_3\in S$. According to Lemma \ref{lemma: key} with $(a,b,c,f)=(2,6,3,3)$, we have $e\in \{1,4,5\}$ such that $(e,3)\in E$. Since $\lambda_2\neq\lambda_3$, then $e\neq 1$. Without loss of generality, we let $e=5$. We have either $\lambda_5=\lambda_3$ or $x_2 x_5 x_6\in S$. Both are impossible.

Finally, if $m=x_1x_5x_6$, then $S=\{x_1^2x_2,x_2^2x_3,x_3^3,x_4^2x_3,x_5^2x_3,x_2x_4x_5,x_1x_5x_6,x_6^2x_2\}$
is large and $G=\{\diag[\omega:\omega^2:1:\omega^2\zeta:\zeta:\omega^3\zeta]|\omega^4=\zeta^2=1\}$, which is case (6).
We completed the discussion for $s=3$.

\smallskip

If $s=2$, we consider the number of self-loops in $E$. If this number is at least $4$, we claim that $S$ is simple. Assume we have $(3,3),(4,4),(5,5),(6,6)\in E$. There exist $(1,i_1),(2,i_2)\in E$. Since $s=2$ and in order to make $S$ non-simple, without loss of generality, we let $i_1=i_2=3$. Then, according to Lemma \ref{lemma: key} with $(a,b,c,f)=(1,2,3,3)$, there exists $e\in\{4,5,6\}$ such that $(e,3)\in E$. Let $e=4$, then $\lambda_4=\lambda_3$ which implies $(2,4)\in E$ and this makes $S$ simple. Next we assume that the number of self-loops in $E$ is at most $3$.

Suppose there exist exactly $3$ self-loops, then we assume $(1,i_1)$, $(2,i_2)$, $(3,i_3)$, $(4,4)$, $(5,5)$, $(6,6)\in E$ and $\lambda_u\neq\lambda_v$ for all $u\in[3],v\in\{4,5,6\}$. If $i_1,i_2,i_3$ are distinct, then $S$ is simple. Otherwise, we assume (without loss of generality) that $i_1=i_2=4$. According to Lemma \ref{lemma: key} with $(a,b,c,f)=(1,2,4,4)$, there exists $e\in\{3,5,6\}$ such that $(e,4)\in E$. If $e=3$, then since $\lambda_3\neq\lambda_4$, hence we must have $x_1x_2x_3\in S$. Now, $S=\{x_1^2x_4,x_2^2x_4,x_3^2x_4,x_1x_2x_3,x_4^3,x_5^3,x_6^3\}$ is large and $G=\{\diag[\alpha:\beta:\alpha\beta:1:\omega:\zeta]|\alpha^2=\beta^2=\omega^3=\zeta^3=1\}$, which is case (7). Next, (without loss of generality) we let $e=5$, then $\lambda_4=\lambda_5$, which implies $(2,5)\in E$. In order not to make $S$ simple, we have $(3,6)\notin E$. In this case, wheneer $(3,4)\in E$ or $(3,5)\in E$,  it becomes the same case as $e=3$, which is already discussed.

Suppose there exist exactly $2$ self-loops, then we assume $(1,5),(2,5),(3,i_3),(4,i_4),(5,5),$
$(6,6)\in E$ and $\lambda_u\neq\lambda_v$ for all $u\in[4],v\in\{5,6\}$. According to Lemma \ref{lemma: key} with $(a,b,c,f)=(1,2,5,5)$, there exists $e\in\{3,4,6\}$ such that $(e,5)\in E$. If $e=3$, from $\lambda_3\neq\lambda_5$ we must have $x_1x_2x_3\in S$. Next we consider the cases $i_4=5,6$ separately.

If $(4,6)\in E$, then $S=\{x_1^2x_5,x_2^2x_5,x_3^2x_5,x_1x_2x_3,x_5^3,x_4^2x_6,x_6^3\}$ is large, and $G=\{[\alpha:\beta:\alpha\beta:\omega:1:\omega^4]|\alpha^2=\beta^2=\omega^6=1\}$, which is case (8).

If $(4,5)\in S$ and $\lambda_5\neq\lambda_6$, then according to Lemma \ref{lemma: key} with $(a,b,c,f)=(1,4,5,5)$, there exists $e_2\in\{2,3,6\}$ such that $(e_2,5)\in E$. Since $\lambda_5\neq \lambda_6$, then $e_2\neq 6$. If $e_2=2$, then from $\lambda_2\ne \lambda_5$ we must have $x_1x_2x_4\in S$. Combining with $x_1 x_2 x_3\in S$, we have $\lambda_3=\lambda_4$. According to Lemma \ref{lemma: main criterion}, there exists $m\in S$ such that $m\in M(2,3,4;1,5)$. Since $\lambda_3=\lambda_4$, we may assume $\deg_4 m=0$. From $\lambda_2,\lambda_3\ne\lambda_5$ we have $m\notin\{x_2^2x_3,x_3^2x_2\}$. Since $\lambda_5\neq\lambda_6$, then $m\notin\{x_2^2x_6,x_3^2x_6\}$. Finally, since $\lambda_1\neq\lambda_6$, then $m\neq x_2x_3x_6$. Now, there is no choice left for $m$, which is a contradiction.

The case $e=4$ is symmetric to $e=3$.
Next, If $e=6$, then $\lambda_6=\lambda_5$, which implies that $(a,b)\in E$ for all $a\in[4],b\in\{5,6\}$. According to Lemma \ref{lemma: main criterion}, there exists $m\in S$ such that $m\in M(1,2,3,4;5,6)$. We assume without loss of generality that $m\in\{x_1^3,x_1^2x_2,x_1x_2x_3\}$. The former two choices are excluded because $\lambda_1,\lambda_2\neq\lambda_5$, and the last choice was already discussed in the case $e=3$.

Suppose there exists exactly one self-loop, then we assume $(u,6)\in E$ for all $u\in[6]$ and $\lambda_v\neq\lambda_6$ for all $v\in[5]$. According to Lemma \ref{lemma: key} with $(a,b,c,f)=(1,2,6,6)$, there exists $e\in \{3,4,5\}$ such that $x_1x_2x_e\in S$. Without loss of generality, we let $e=3$. Using Lemma \ref{lemma: key} again with $(a,b,c,f)=(1,4,6,6)$, there exists $e_2\in\{2,3,5\}$ such that $x_1 x_4 x_{e_2}\in S$.

If $e_2=2$, then both $x_1 x_2 x_3$ and $x_1 x_2 x_4$ belong to $S$, hence $\lambda_3=\lambda_4$. Using Lemma \ref{lemma: key} with $(a,b,c,f)=(3,4,6,6)$, then there exists $e_3\in\{1,2,5\}$ with $x_3x_4x_{e_3}\in S$. From $\lambda_3=\lambda_4$ and $x_3^2 x_6\in S$ we must have $\lambda_{e_3}=\lambda_6$, which is a contradiction. The case $e_2=3$ similarly leads to contradiction.

If $e_2=5$, then $x_1x_4x_5\in S$. According to Lemma \ref{lemma: key} with $(a,b,c,f)=(2,4,6,6)$, we obtain $x_2x_4x_{e_3}\in S$ for certain $e_3\in\{1,3,5\}$. If $e_3=1$, then we are again in the situation $e_2=2$. If $e_3=3$, then from $x_1x_2x_3, x_2 x_3 x_4\in S$ we have $\lambda_1=\lambda_4$. Combining with $x_1x_4x_5, x_1^2 x_6\in S$, we have $\lambda_5=\lambda_6$, which is a contradiction. The case $e_3=5$ similarly leads to contradiction. The proof of the theorem is now completed.
\end{proof}

\section{Quartic Surfaces}
\label{section: quartic}

\begin{thm}
\label{theorem: main2}
Suppose $G$ is an abelian group with a faithful and liftable action on a smooth quartic surface $X$. Choose a coordinate system $x_1, \cdots, x_4$ for $\PP^3$ such that the action of $G$ is diagonal. Let $I=(G, \lambda, \lambda_1, \cdots, \lambda_4)$ be the associated data. Suppose $[I]$ is maximal in $\calP_{4,4}$, then either $S=S_I$ is simple, or we have the following possibilities (up to a permutation of coordinates $x_1, \cdots, x_4$):
\begin{enumerate}[(1)]
\item $G=\{\diag[1:\omega:\omega^4:\omega^3]\big{|}\omega^6=1\}\cong C_6$, \\ $S=\{x_1^3x_2,x_1^2x_3x_4,x_1x_2x_4^2,x_2^3x_3,x_3^3x_2,x_4^3x_3\}$,
\item $G=\{\diag[1:\omega:\omega^4:\omega^5]\big{|}\omega^6=1\}\cong C_6$, \\ $S=\{x_1^3x_2,x_1x_2^2x_4,x_1x_3^2x_4,x_2^3x_3,x_3^3x_2,x_4^3x_3\}$,
\item $G=\{\diag[\omega:\omega^6:1:\omega^4]\big{|}\omega^9=1\}\cong C_9$, \\
$S=\{x_1^3 x_2, x_2^3 x_3, x_3^4, x_4^3 x_2, x_1x_3x_4^2\}$,
\item $G=\{\diag[\omega:\omega^6:1:\omega^3]\big{|}\omega^9=1\}\cong C_9$, \\
$S=\{x_1^3 x_2, x_2^3 x_3, x_3^4, x_4^3 x_2, x_2^2x_4^2,x_2x_3^2x_4\}$, and
\item $G=\{\diag[\omega^4:1:\omega^8:\omega^3]\big{|}\omega^{12}=1\}\cong C_{12}$, \\ $S=\{x_1^3x_2,x_1^2x_3^2,x_1x_2^2x_3,x_2^4,x_3^3x_2,x_4^4\}$.
\end{enumerate}
\end{thm}
\begin{proof}
We write $s=s_I,c=c_I$ and $E=E_S$. We first consider the case $c>0$.

By Proposition \ref{proposition: loop}, we have $s+c\leqslant 4$. It is easy to check that, if $s=c+1$, then $S$ must be simple. Suppose $s>c+1$, then the only choice is $s=3,c=1$. In this case, we may assume (without loss of generality) that $\{x_1^3 x_2, x_2^3 x_3, x_3^3 x_2, x_4^3 x_3\}\subseteq S$ and $\lambda_1,\lambda_2,\lambda_3$ are distinct. Take an element $g=\diag[1:\omega:\rho:\omega\zeta]\in G$. Then $\omega=\omega^3 \rho=\rho^3 \omega=\omega^3\zeta^3\zeta$. This implies that $\rho=\omega^{-2}$, $\zeta^3=1$ and $\omega^6=1$. Therefore,
$$G\leqslant \{\diag[1:\omega:\omega^4:\omega\zeta]\big{|}\omega^6=\zeta^3=1\}\cong C_6\oplus C_3.$$

Since $\lambda_1\ne \lambda_3$, there exists an element $\diag[1:\omega:\omega^4:\omega\zeta]\in G$ with $\omega^2\ne 1$. Similarly, since $\lambda_2\ne \lambda_3$, there exists an element $\diag[1:\omega:\omega^4:\omega\zeta]\in G$ with $\omega^3\ne 1$. Therefore, there exists $\diag[1:\omega_0:\omega_0^4:\omega_0\zeta_0]\in G$ with $\omega_0$ a primitive $6$-root. This implies that $G\cong C_6\oplus C_3$ or $C_6$.

If $G\cong C_6\oplus C_3$, then $G=\{\diag[1:\omega:\omega^4:\omega\zeta]\big{|}\omega^6=\zeta^3=1\}$. By straightforward calculation, we have $S=\{x_1^3 x_2, x_2^3 x_3, x_3^3 x_2, x_4^3 x_3\}$, which is small due to lack of a monomial in $M(1,3;2)$.

Next we assume $G\cong C_6$. Since $\omega_0$ is a primitive $6$-root, and we must have $\zeta_0\in\{1,\omega_0^2,\omega_0^4\}$. If $\zeta_0=1$, then $G=\{\diag[1:\omega:\omega^4:\omega]\big{|}\omega^6=1\}$ and  $S=\{x_1^3x_2,x_1^3x_4,x_2^3x_3,x_2^2x_3x_4,x_3^3x_2,$
$x_2x_3x_4^2,x_3^3x_4,x_4^3x_3\}$, which is small due to lack of monomial in $M(2,4;3)$.

If $\zeta_0=\omega_0^2$, then $G=\{\diag[1:\omega:\omega^4:\omega^3]\big{|}\omega^6=1\}$ and $S=\{x_1^3x_2,x_1^2x_3x_4,x_1x_2x_4^2,x_2^3x_3,x_3^3x_2,$
$x_4^3x_3\}$, which is large because the polynomial $F=2x_1^3x_2+x_1^2x_3x_4+x_1x_2x_4^2+x_2^3x_3+x_3^3x_2+x_4^3x_3$ is smooth (In this case, if all coefficients equal to $1$, then we obtain a polynomial that is not smooth). This gives rise to case (1).

If $\zeta_0=\omega_0^4$, then $G=\{\diag[1:\omega:\omega^4:\omega^5]\big{|}\omega^6=1\}$ and $S=\{x_1^3x_2,x_1x_2^2x_4,x_1x_3^2x_4,x_2^3x_3,x_3^3x_2,$
$x_4^3x_3\}$, which is large and gives rise to case (2).

\smallskip

From now on, we assume $c=0$. We may assume $\{x_1^3x_2, \cdots, x_{s-1}^3 x_s,x_s^4\}\subseteq S$ and $\lambda_1,\cdots, \lambda_s$ distinct. If $s=4$, then $S$ is simple. Next we assume $s\leqslant 3$. Suppose $s=1$, then $S$ is simple.

Suppose $s=3$. By Lemma \ref{lemma: main criterion}, we have $i\in[4]$ such that $(4,i)\in E$. To make $S$ non-simple, $i\in\{2,3\}$.

If $(4,2)\in E$, then by similar calculation we obain $G\leqslant \{\diag[\omega:\omega^6:1:\omega\zeta]\big{|}\omega^9=\zeta^3=1\}\cong C_9\oplus C_3.$
Since $\lambda_2\ne\lambda_3$, there exists an element $\diag[\omega:\omega^6:1:\omega \zeta]\in G$ with $\omega^3\ne 1$. Thus there exists an element $\diag[\omega_0:\omega_0^6:1:\omega_0 \zeta_0]\in G$ with $\omega_0$ a primitive $9$-root. Therefore, we have either $G\cong C_9\oplus C_3$ or $C_9$.

If $G\cong C_9\oplus C_3$, then $G=\{\diag[\omega:\omega^6:1:\omega\zeta]\big{|}\omega^9=\zeta^3=1\}$ and $S=\{x_1^3 x_2, x_2^3 x_3, x_3^4, x_4^3 x_2\}$, which is small due to lack of a monomial in $M(1,4;2)$.

Next we assume $G\cong C_9$. Since $\omega_0$ is a primitive $9$-root, we have $\zeta_0\in\{1,\omega_0^3,\omega_0^6\}$.

If $\zeta_0=1$, then $G=\{\diag [\omega:\omega^6:1:\omega]\big{|}\omega^9=1\}$ and  $S=\{x_1^3 x_2, x_2^3 x_3, x_3^4, x_4^3 x_2, x_1^2x_2x_4,x_1x_2x_4^2\}$, which is small due to lack of a monomial in $M(1,4;2)$.

If $\zeta_0=\omega_0^3$, then $G=\{\diag[\omega:\omega^6:1:\omega^4]\big{|}\omega^9=1\}$ and  $S=\{x_1^3 x_2, x_2^3 x_3, x_3^4, x_4^3 x_2, x_1x_3x_4^2\}$, which is large and gives rise to case (3).

If $\zeta_0=\omega_0^6$, then $G=\{\diag[\omega:\omega^6:1:\omega^7]\big{|}\omega^9=1\}$ and  $S=\{x_1^3 x_2, x_2^3 x_3, x_3^4, x_4^3 x_2, x_1^2x_3x_4\}$, which is symmetric to the case (3).

If $(4,3)\in E$, then we have $G\leqslant \{\diag [\omega:\omega^6:1:\omega^6\zeta]\big{|}\omega^9=\zeta^3=1\}\cong C_9\oplus C_3$, and there exists $\diag [\omega_0:\omega_0^6:1:\omega_0^6\zeta_0]\in G$ with $\omega_0$ a primitive $9$-root. We have either $G\cong C_9\oplus C_3$ or $C_9$.

If $G\cong C_9\oplus C_3$, then $S=\{x_1^3 x_2, x_2^3 x_3, x_3^4, x_4^3 x_3\}$, which is small due to lack of a monomial in $M(2,4;3)$.

Next we assume $G\cong C_9$. We have $\zeta_0\in\{1,\omega_0^3,\omega_0^6\}$.

If $\zeta_0=1$, then $G=\{\diag[\omega:\omega^6:1:\omega^6]\big{|}\omega^9=1\}$ and  $S=\{x_1^3 x_2, x_2^3 x_3, x_3^4, x_4^3 x_3, x_1^3x_4, x_2^2x_3x_4,$
$x_2x_3x_4^2\}$, which is small due to lack of a monomial in $M(2,4;3)$.

If $\zeta_0=\omega_0^3$, then $x_4^4\in S$ and $S$ is simple.

If $\zeta_0=\omega_0^6$, then $G=\{\diag[\omega:\omega^6:1:\omega^3]\big{|}\omega^9=1\}$ and  $S=\{x_1^3 x_2, x_2^3 x_3, x_3^4, x_4^3 x_2, x_2^2x_4^2,x_2x_3^2x_4\}$, which is large and gives rise to the case (4).

\smallskip

Finally we discuss the case $s=2$ and $c=0$. We have now $\{x_1^3 x_2, x_2^4\}\subseteq S$ and $\lambda_1\ne \lambda_2$. There exists $i_3, i_4\in[4]$ such that $(3,i_3), (4, i_4)\in E$. By the maximality of $s$, we have $(3,1), (4,1)\notin E$.

Suppose $(3,2), (4,2)\notin E$. Then $i_3, i_4\in \{3,4\}$, which implies that $S$ is simple.

Suppose $(3,2), (4,4)\in E$. If $\lambda_4=\lambda_2$, then $x_3^3 x_4, x_4^4\in S$, which implies that $S$ is simple. If $\lambda_3=\lambda_2$, then $x_3^4, x_4^4\in S$, which also implies that $S$ is simple. Next we assume $\lambda_3,\lambda_4\neq \lambda_2$.

By straightforward calculation, we obtain $G\leqslant \{\diag[\alpha:1:\beta:\gamma]\big{|}\alpha^3=\beta^3=\gamma^4=1\}$. According to Lemma \ref{lemma: main criterion}, there exists a monomial $m\in S$ such that $m\in M(1,3;2)$.

We first suppose $\deg_1m+\deg_3m=4$. From $\lambda_1,\lambda_3\neq \lambda_2$ we have $m=x_1^2x_3^2$. Therefore, for an element $\diag[\alpha:1:\beta:\gamma]\in G$, we have $\alpha^2\beta^2=1$. Since $\alpha^3=\beta^3=1$, we obtain that $\beta=\alpha^2$. Then we can rewrite G as $\{\diag[\omega^4:1:\omega^8:\omega^3]\big{|}\omega^{12}=1\}\cong C_{12}$, and $S=\{x_1^3x_2,x_1^2x_3^2,x_1x_2^2x_3,x_2^4,x_3^3x_2,x_4^4\}$, which is large and gives rise to case (5).

Next we suppose $\deg_1m+\deg_3m=3$, then $\deg_4 m=1$. Without loss of generality, we assume $m=x_1^2 x_3 x_4$. Then, for an element $\diag[\alpha:1:\beta:\gamma]\in G$, we have $\alpha^2\beta\gamma=1$. Taking fourth power on both side, we conclude that $\alpha^2\beta=1$, which implies that $\alpha=\beta$ and $\gamma=1$. Thus $\lambda_4=\lambda_2$, which is a contradiction.

The case $(3,3), (4,2)\in E$ is symmetric to the case $(3,2), (4,4)\in E$ as discussed above.

Suppose $(4,3), (3,2)\in E$ and $(4,4)\notin E$. We must have $\lambda_3\ne \lambda_4$. Since $(2,2)\in E$, we must have  $\lambda_4\neq \lambda_2$. If $\lambda_3\ne \lambda_2$, then the fact $x_4^3 x_3, x_3^3 x_2, x_2^4\in S$ contradicts to the maximality of $s$. If $\lambda_3=\lambda_2$, then $x_4^3 x_3, x_3^4\in S$, which implies that $S$ is simple.

The case $(3,4), (4,2)\in E$ and $(3,3)\notin E$ is symmetric to the above case.

We finally need to discuss the situation $(3,2), (4,2)\in E$ and $(4,3), (3,4), (3,3), (4,4)\notin E$. Then $\lambda_1,\lambda_3,\lambda_4\neq \lambda_2$, and there does not exist any edge $(a,b)\in E$ such that $a,b\in[4]\backslash \{2\}$. We have $G\leqslant \{\diag[\alpha:1:\beta:\gamma]\big{|}\alpha^3=\beta^3=\gamma^3=1\}\cong C_3^3$.  According to Lemma \ref{lemma: main criterion}, there exists a monomial $m_1\in S$ such that $m_1\in M(1,3;2)$. Without loss of generality, we assume that $\deg_1 m\geqslant \deg_3 m$, then $m_1\in\{x_1^2x_3^2,x_1^2x_3x_4\}$. In any case, there exists another $m_2\in S$ such that $m_2\in M(3,4;2)$. The choices are $m_2\in\{x_3^2x_4^2,x_1x_3^2x_4,x_1x_3x_4^2\}$. For an element $\diag[\alpha:1:\beta:\gamma]\in G$, the two monomials $m_1, m_2$ impose two equations on $\alpha, \beta, \gamma$, which leads to $G\cong C_3$.


Without loss of generality, the two choices are $G=\{\diag[\omega:1:\omega:\omega]\big{|}\omega^{3}=1\}$ and $G=\{\diag[\omega:1:\omega:\omega^2]\big{|}\omega^{3}=1\}$. But both cannot happen due to lack of a monomial in $M(1,3;2)$.
\end{proof}

\bibliography{reference}

\end{document}